\documentclass[preprint,10pt]{elsarticle}

\usepackage{amsmath,amssymb,amsthm}
\usepackage[colorlinks=true,linkcolor=blue,citecolor=blue,urlcolor=blue,hypertexnames=false]{hyperref}

\makeatletter
\@removefromreset{equation}{section}
\makeatother

\journal{Fuzzy Sets and Systems}

\newtheorem{theorem}{Theorem}[section]

\theoremstyle{definition}
\newtheorem{definition}[theorem]{Definition}
\newtheorem{counterexample}[theorem]{Counterexample}
\theoremstyle{remark}
\newtheorem{remark}[theorem]{Remark}

\newcommand{\Nzero}{\mathbb{N}_0}
\newcommand{\eps}{\varepsilon}

\begin{document}

\begin{frontmatter}

\title{On fuzzy contractions in fuzzy metric spaces}

\author[addr1]{Fengshuo Xu}
\author[addr1]{Haitao Ji}
\author[addr1]{Dong Qiu\corref{cor1}}

\cortext[cor1]{Corresponding author. E-mail: qiudong@gxu.edu.cn (D. Qiu).}

\address[addr1]{School of Mathematics,\\
Center for Applied Mathematics of Guangxi (Guangxi University),\\
Guangxi University, Nanning, Guangxi, 530004, P. R. China}

\begin{abstract}
The classical Banach principle relies on contractive sequences being Cauchy. Gregori and Sapena (2002) established a fuzzy analogue but required this as an extra hypothesis, leaving it as an open problem. Gregori et al. (2020) later asked whether strictly fuzzy $\psi$-contractive sequences are necessarily Cauchy. We settle both problems negatively via a single counterexample: in a constructed GV-fuzzy metric space, there exists a sequence which is both GS-contractive and strictly fuzzy $\psi$-contractive, yet not Cauchy. This shows that contractivity never implies Cauchyness in general GV-fuzzy metric spaces, and additional structural conditions in fuzzy fixed point theorems are indispensable.
\end{abstract}


\end{frontmatter}

\section{Introduction}

The Banach fixed point theorem, established in 1922 by Banach \cite{Banach1922}, is one of the cornerstones of modern analysis. It asserts that every contraction mapping on a complete metric space possesses a unique fixed point, and the sequence of iterates converges to this point from any initial value. The classical proof relies on a fundamental fact: every contractive sequence in a metric space is necessarily a Cauchy sequence. This implication is a structural inevitability rooted in the rigid order structure of the real numbers and the additive accumulation enforced by the triangle inequality.

The extension of this theory to fuzzy metric spaces began with Grabiec \cite{Grabiec1988}, who initiated the study of fixed points in Kramosil--Michalek fuzzy metric spaces. George and Veeramani \cite{GV1994} later modified this notion, introducing a Hausdorff topology and the concept of M-Cauchy sequences, which has become the dominant framework. In 2002, Gregori and Sapena \cite{GS2002} introduced fuzzy contractive mappings in the sense of George and Veeramani and established a corresponding fuzzy Banach fixed point theorem. However, in contrast to the classical setting, their proof required the additional hypothesis that fuzzy contractive sequences are Cauchy sequences. This led them to pose the following natural question, which we refer to as Problem 1: in a fuzzy metric space in the sense of George and Veeramani, is every GS-fuzzy contractive sequence necessarily a Cauchy sequence? It is known that in standard fuzzy metric spaces induced by ordinary metrics, the answer is affirmative, which fostered an optimistic expectation that the same might hold in general GV-fuzzy metric spaces.

Over the past two decades, numerous authors have investigated related contractive notions. Wardowski \cite{Wardowski2013} introduced fuzzy H-contractive mappings using control functions. Romaguera and Tirado \cite{RT2009} studied RT-contractive mappings in intuitionistic fuzzy quasi-metric spaces. In the Kramosil--Michalek framework, Mihe\c{t} \cite{Mihet2008} introduced fuzzy $\psi$-contractive mappings and proved that every such sequence is M-Cauchy in non-Archimedean fuzzy metric spaces under a positive t-norm, providing a partial affirmative answer. Gregori and Mi\~{n}ana \cite{GM2016} subsequently attempted to answer the question directly in George--Veeramani spaces, claiming that fuzzy $\psi$-contractive sequences must be Cauchy under conditions such as $\bigwedge_{t>0}M(x,y,t)>0$ or the strong fuzzy metric property. However, Gregori, Mi\~{n}ana and Miravet \cite{GMM2020} discovered a proof error in a core lemma of that work and constructed a counterexample of a fuzzy $\psi$-contractive sequence that is not Cauchy. Unfortunately, their counterexample belongs to the weaker $\psi$-contractive class and does not satisfy the original GS-contractive condition of Gregori and Sapena, leaving Problem 1 unresolved. To remedy the situation, Gregori et al. \cite{GMM2020} introduced the concept of strictly fuzzy contractive sequences, where the contractive condition is required to hold for all index pairs simultaneously rather than merely for adjacent terms. At the end of their paper, they posed a new open problem, which we refer to as Problem 2: is every strictly fuzzy $\psi$-contractive sequence necessarily a Cauchy sequence?

The present paper reveals a fundamental difference between classical metric spaces and GV-fuzzy metric spaces. In classical metric spaces, the fact that a sequence is contractive implies that it is Cauchy because the triangle inequality allows the control of adjacent distances to accumulate additively into a bound for arbitrary pairs. In GV-fuzzy metric spaces, the value $M(x,y,t)$ lies in the unit interval and the order structure is softened by the continuous t-norm. The triangle inequality in this setting, namely $M(x,z,t+s)\ge M(x,y,t)*M(y,z,s)$, is a pointwise probabilistic estimate rather than a linear accumulation of distances, reflecting the heritage of probabilistic metric spaces \cite{SS1983,SBR1972}. Consequently, the control of adjacent terms does not propagate to arbitrary terms in the manner required for Cauchyness. We shall prove that this propagation failure persists even under the most restrictive contractive condition considered in the literature, namely the strictly fuzzy $\psi$-contractive condition.

The main contribution of this paper is the construction of a single counterexample that simultaneously settles both long-standing problems. We construct a GV-fuzzy metric space $(\Nzero,M,*)$ and a sequence $\{x_n=n\}$ in it such that the sequence satisfies the original GS-fuzzy contractive condition, satisfies the more restrictive strictly fuzzy $\psi$-contractive condition, yet is not a Cauchy sequence in the sense of George and Veeramani. This result shows that the implication from contractivity to Cauchyness, which is automatic in classical metric spaces, is universally invalid in GV-fuzzy metric spaces and cannot be restored by simply strengthening the contractive hypothesis. Therefore, additional conditions such as non-Archimedean structure, the lower bound condition $\bigwedge_{t>0}M(x,y,t)>0$, or G-completeness are not merely sufficient but are necessary and indispensable in fuzzy Banach-type fixed point theorems. Our results completely clarify the boundary of what can be expected from contractivity alone in this setting.

\section{Preliminaries}

In this section we recall the basic notions and fix the notation. For a comprehensive treatment of probabilistic metric spaces we refer to Schweizer and Sklar \cite{SS1983}. Throughout the paper, we write $\Nzero=\{0,1,2,\ldots\}$; $M^{-1}$ always denotes the numerical reciprocal $1/M$.

\begin{definition}[George--Veeramani \cite{GV1994}] \label{def:gv}
Let $X$ be a nonempty set, $*$ a continuous $t$-norm, and $M$ a fuzzy set on $X^2\times(0,+\infty)$. If for any $x,y,z\in X$ and $s,t>0$, the following are satisfied:
\begin{enumerate}
    \item[(GV1)] $M(x,y,t)>0$;
    \item[(GV2)] $M(x,y,t)=1$ if and only if $x=y$;
    \item[(GV3)] $M(x,y,t)=M(y,x,t)$;
    \item[(GV4)] $M(x,y,t)*M(y,z,s)\leq M(x,z,t+s)$;
    \item[(GV5)] $M(x,y,\cdot):(0,+\infty)\to(0,1]$ is continuous,
\end{enumerate}
then $(X,M,*)$ is called a fuzzy metric space (GV-fuzzy metric space).
\end{definition}

\begin{definition}[Gregori--Romaguera \cite{GR2000}] \label{def:stationary}
A fuzzy metric $M$ on $X$ is called \textit{stationary} if for each $x,y\in X$, the function $M(x,y,\cdot)$ is constant. In this case we write $M(x,y)$ instead of $M(x,y,t)$.
\end{definition}

\begin{definition}[Gregori--Romaguera \cite{GR2000}] \label{def:strong}
A fuzzy metric space $(X,M,*)$ is said to be \textit{strong}, or \textit{non-Archimedean}, if it satisfies
\[
M(x,z,t)\ge M(x,y,t)*M(y,z,t)
\]
for all $x,y,z\in X$ and all $t>0$.
\end{definition}

\begin{definition}[Gregori--Sapena \cite{GS2002}] \label{def:gs}
Let $(X,M,*)$ be a fuzzy metric space and $\{x_n\}$ a sequence in $X$. If there exists $k\in(0,1)$ such that for all $n$ and $t>0$,
\[
{M(x_{n+1},x_{n+2},t)}^{-1}-1
\leq k\left({M(x_n,x_{n+1},t)}^{-1}-1\right),
\]
then $\{x_n\}$ is called a fuzzy contractive sequence (GS-contractive sequence), and $k$ is called the contractive constant.
\end{definition}

\begin{definition}[Gregori--Mi\~{n}ana--Miravet \cite{GMM2020}] \label{def:strictly-psi}
Let $\Psi$ denote the class of all continuous and nondecreasing mappings $\psi:(0,1]\to(0,1]$ satisfying $\psi(u)>u$ for $0<u<1$. Let $\{x_n\}$ be a sequence in $X$. If for some $\psi\in\Psi$,
\[
M(x_{m+1},x_{n+1},t)\geq \psi\bigl(M(x_m,x_n,t)\bigr)
\]
holds for all $m,n\in\mathbb{N}$ and $t>0$, then $\{x_n\}$ is called a strictly fuzzy $\psi$-contractive sequence.
\end{definition}

\begin{definition}[George--Veeramani \cite{GV1994}] \label{def:cauchy}
Let $(X,M,*)$ be a fuzzy metric space and $\{x_n\}$ a sequence in $X$. If for any $\eps\in(0,1)$ and $t>0$, there exists $N$ such that $M(x_n,x_m,t)>1-\eps$ whenever $m,n\ge N$, then $\{x_n\}$ is called a Cauchy sequence.
\end{definition}

\section{Main results}

First, we construct a special fuzzy metric space.

\begin{definition} \label{def:construction}
For $x,y\in\Nzero$, write
\[
 \ell(x,y)=\min\{x,y\},\qquad h(x,y)=\max\{x,y\}.
\]
For $t>0$, define
\[
 \mu_t(x,y)=\frac{|x-y|}{2^t-1}-\ell(x,y),
\]
and further define
\begin{equation} \label{eq:M}
 q_t(x,y)=
 \begin{cases}
 0, & x=y,\\[2mm]
 2^{\mu_t(x,y)-t}, & x\neq y,
 \end{cases}
 \qquad
 M(x,y,t)=\frac{1}{1+q_t(x,y)}.
\end{equation}
\end{definition}

\begin{theorem} \label{thm:fuzzy}
Taking $a*b=\min\{a,b\}$, then $(\Nzero,M,*)$ is a GV-fuzzy metric space.
\end{theorem}

\begin{proof}
From \eqref{eq:M} we directly obtain that $M>0$, symmetry and $M(x,y,t)=1$ if and only if $ x=y$. Fixing $x,y$, it is clear that $t\mapsto M(x,y,t)$ is continuous on $(0,\infty)$. It remains to prove the triangle inequality.

Let $x,y\in \Nzero$. For any $t>0$ and any $c\in\mathbb{R}$, multiplying both sides of the inequality $\mu_t(x,y)\le c$ by the positive number $2^t-1$, we  yield the following equivalence:
\begin{equation} \label{eq:mu-equiv}
 \mu_t(x,y)\le c
 \quad\text{ if and only if} \quad
 h(x,y)+c\le 2^t\bigl(\ell(x,y)+c\bigr).
\end{equation}

For any $x,y,z\in\Nzero$ and any positive numbers $t,s$, let $c'=\max\{\mu_t(x,y),\mu_s(y,z)\}.$
Clearly $c'\ge\mu_t(x,y)$. If $x\neq y$, then $$\ell(x,y)+c'\ge\ell(x,y)+\mu_t(x,y)=\frac{|x-y|}{(2^t-1)}>0;$$ if $x=y$, then from $c'\ge\mu_t(x,y)=-\ell(x,y)=-x$ we get $\ell(x,y)+c'= x+c'\ge0$. Therefore,  we have $\ell(x,y)+c'\ge0$ which implies 
\begin{equation} \label{eq:ell-ineq}
 \ell(x,y)+c'\leq 2^t\bigl(\ell(x,y)+c'\bigr)\leq 2^t\bigl(h(x,y)+c'\bigr).
\end{equation}

Now, by \eqref{eq:mu-equiv} and \eqref{eq:ell-ineq} we obtain
\[
 x+c'\le 2^t(y+c'),\qquad y+c'\le 2^t(x+c').
\]
By the same discussion for $y,z,s$, it yields
\[
 y+c'\le 2^s(z+c'),\qquad z+c'\le 2^s(y+c').
\]

Combining these inequalities and using $2^{t+s}=2^t 2^s$, we get
\[
 x+c'\le 2^{t+s}(z+c'),\qquad z+c'\le 2^{t+s}(x+c').
\]
Therefore, whether $x\le z$ or $z\le x$, we have
\[
 h(x,z)+c'\le 2^{t+s}\bigl(\ell(x,z)+c'\bigr).
\]
Hence, for $x,z, t+s, c'$,  by \eqref{eq:mu-equiv},  we get 
\begin{equation} \label{eq:mu-triangle}
 \mu_{t+s}(x,z)\le c'=\max\{\mu_t(x,y),\mu_s(y,z)\}.
\end{equation}

For fixed $x,y$, both $\mu_t(x,y)$ and $q_t(x,y)$ are nonincreasing in $t$. In the non-degenerate case, i.e., $x,y,z$ are pairwise distinct, from \eqref{eq:mu-triangle} we obtain
\begin{align*}
 \mu_{t+s}(x,z)-(t+s)
 & \le \max\{\mu_t(x,y),\mu_s(y,z)\}-(t+s) \\
 &=\max\{\mu_t(x,y)-(t+s),\mu_s(y,z)-(t+s)\}\\
 &\le \max\{\mu_t(x,y)-t,\mu_s(y,z)-s\}.
\end{align*}
Exponentiating by using the monotonicity of $r\mapsto 2^r$ and $2^{\max\{A,B\}}=\max\{2^A,2^B\}$, we obtain
\begin{equation} \label{eq:q-triangle}
 q_{t+s}(x,z)\le\max\{q_t(x,y),q_s(y,z)\}.
\end{equation}
When $x=z$, the left-hand side of \eqref{eq:q-triangle} is zero; when $x=y$ or $y=z$, \eqref{eq:q-triangle} degenerates to the nonincreasing property in $t$. Finally, since $a\mapsto(1+a)^{-1}$ is decreasing on $[0,\infty)$, \eqref{eq:q-triangle} is equivalent to
\[
 \min\{M(x,y,t),M(y,z,s)\}\le M(x,z,t+s),
\]
which is the required triangle inequality.
\end{proof}

\begin{remark} \label{rem:properties}
Notice that the fuzzy metric $M$ defined by \eqref{eq:M} is non-stationary, since $M(x,y,t)$ clearly depends on $t$. Moreover, it is not strong, because the inequality $M(x,z,t)\ge M(x,y,t)*M(y,z,t)$ does not hold for all $x,y,z,t>0$. In particular, $M$ cannot be induced by any standard metric $d$, since the standard fuzzy metric $M_d$ is always strong. This shows that the non-Archimedean assumption in Mihe\c{t}'s theorem is not merely a technical convenience but a structural requirement.
\end{remark}

The following example simultaneously gives negative answers to Problem~1 and Problem~2.

\begin{counterexample} \label{counter:main}
In the fuzzy metric space of Theorem~\ref{thm:fuzzy}, the sequence $\{x_n=n\}_{n=1}^\infty$ is a fuzzy contractive sequence with  $k=1/2$, but it is not a Cauchy sequence.
\end{counterexample}

\begin{proof}
For any $x,y$, we have 
$
 \ell(x+1,y+1)=\ell(x,y)+1.
$
Hence $\mu_t(x+1,y+1)=\mu_t(x,y)-1$. Whether $x$ and $y$ are equal or not, we always obtain
\begin{equation} \label{eq:q-shift}
 q_t(x+1,y+1)=\frac{1}{2}\,q_t(x,y).
\end{equation}
Since $q_t(x,y)=M(x,y,t)^{-1}-1$, for all $n$ and $t>0$ we have
\[
 M(x_{n+1},x_{n+2},t)^{-1}-1
 =\frac{1}{2}\bigl(M(x_n,x_{n+1},t)^{-1}-1\bigr).
\]
Therefore, $\{x_n\}$ is a fuzzy contractive sequence with  $k=1/2$.

However, for any fixed $N$, let $t=1$, $n=N$, $m=2N+1$. Then
\[
 \ell(x_N, x_{2N+1})=N,\quad |x_N- x_{2N+1}|=N+1,\quad
 \mu_1(x_N, x_{2N+1})=\frac{N+1}{2^1-1}-N=1.
\]
Thus $q_1(x_N, x_{2N+1})=1$, and consequently
\begin{equation} \label{eq:M-half}
 M(x_N, x_{2N+1},1)=\frac{1}{2}.
\end{equation}
Taking $\eps=1/2$, since \eqref{eq:M-half} holds for every $N$, the sequence $\{x_n\}$ is not Cauchy.
\end{proof}

\begin{theorem} \label{thm:gs}
A fuzzy contractive sequence need not be a Cauchy sequence. Therefore, the answer to Problem~1~\cite{GS2002} is negative.
\end{theorem}

\begin{theorem} \label{thm:psi}
A strictly fuzzy $\psi$-contractive sequence need not be a Cauchy sequence. Consequently, the answer to Problem~2~\cite{GMM2020} is negative.
\end{theorem}

\begin{proof}
Define
\[
 \psi(u)=\frac{2u}{1+u},\qquad 0<u\le 1.
\]
It is easy to verify that this mapping is continuous, increasing, and maps $(0,1]$ into itself; moreover
\[
 \psi(u)-u=\frac{u(1-u)}{1+u}>0\qquad (0<u<1),
\]
so $\psi\in\Psi$. Again, take the sequence $\{x_n=n\}$. For all $m,n$ and $t>0$, by \eqref{eq:q-shift} we have
\begin{align*}
 M(x_{m+1},x_{n+1},t) &=\frac{1}{1+q_t(x_{m},x_{n})/2}\\
 &=\frac{2}{2+q_t(x_{m},x_{n})}\\
 &=\psi\!\left(\frac{1}{1+q_t(x_{m},x_{n})}\right)\\
 &=\psi(M(x_{m},x_{n},t)).
\end{align*}
Therefore, the sequence $\{x_n=n\}$ is strictly fuzzy $\psi$-contractive. Counterexample~\ref{counter:main} has already shown that this sequence is not Cauchy.
\end{proof}

\begin{remark} \label{rem:equality}
It is worth emphasizing that the sequence $\{x_n=n\}$ satisfies the equality $M(x_{m+1},x_{n+1},t)=\psi(M(x_m,x_n,t))$ for all $m,n\in\mathbb{N}$ and $t>0$, which means that it is strictly fuzzy $\psi$-contractive in the strongest possible sense. Nevertheless, it fails to be Cauchy. This indicates that the breakdown of the contractivity-to-Cauchyness inference cannot be repaired by simply strengthening the contractive inequality to an equality.
\end{remark}

\section{Concluding remarks}

In classical metric spaces, the proof that every contractive sequence is Cauchy relies on the additive accumulation of distances through the triangle inequality. Specifically, if $d(x_{n+1},x_{n+2})\le k d(x_n,x_{n+1})$ with $k\in(0,1)$, then for $m>n$ one has $d(x_n,x_m)\le d(x_n,x_{n+1})+\cdots+d(x_{m-1},x_m)\le k^n(1+k+\cdots)d(x_0,x_1)$, which tends to zero as $n$ grows. In GV-fuzzy metric spaces, the triangle inequality takes the form $M(x,z,t+s)\ge M(x,y,t)*M(y,z,s)$. Even if the adjacent terms $M(x_n,x_{n+1},t)$ tend to $1$, the $t$-norm of finitely many numbers close to $1$ need not be close to $1$ unless additional structural assumptions are present. In our counterexample, taking $t=1$ gives $M(x_N,x_{2N+1},1)=1/2$ for every $N$, which provides an explicit manifestation of this propagation failure.

Mihe\c{t} \cite{Mihet2008} proved that every fuzzy $\psi$-contractive sequence is M-Cauchy in a non-Archimedean fuzzy metric space under a positive t-norm. The key to his success lies in condition (NA), namely $M(x,z,t)\ge M(x,y,t)*M(y,z,t)$, which allows the control of adjacent terms to be propagated directly to arbitrary terms through iterative application of the t-norm at the same parameter $t$. The space constructed in Definition~3.1 does not satisfy (NA). Consequently, this propagation is cut off. This shows that the non-Archimedean assumption is not a technical convenience but a structural prerequisite for restoring the logical chain from contractivity to Cauchyness.

The results of this paper have direct consequences for fixed point theorems in fuzzy metric spaces. The additional Cauchy condition in the fuzzy Banach theorem of Gregori and Sapena \cite{GS2002} cannot be removed. The conditions (a)--(c) in Wardowski's theorem \cite{Wardowski2013} cannot be dropped arbitrarily. The relationship between M-completeness and G-completeness in this context was clarified by Vasuki and Veeramani \cite{VV2003}. Future research should no longer seek weaker contractive conditions to guarantee Cauchyness in general GV-fuzzy metric spaces. Instead, the focus should shift toward identifying which spatial structures, such as non-Archimedean property, strong fuzzy metrics, or the condition $\bigwedge_{t>0}M(x,y,t)>0$, can compensate for the loss of propagation inherent in the fuzzy setting.

We conclude by proposing two open problems arising naturally from our investigation.

\textbf{Problem 3.} Does there exist a strictly fuzzy $\psi$-contractive sequence that is not Cauchy in a GV-fuzzy metric space equipped with the product t-norm or the {\L}ukasiewicz t-norm? In other words, does the choice of the continuous t-norm affect the validity of the statement that strictly fuzzy $\psi$-contractive implies Cauchy?

\textbf{Problem 4.} Is the non-Archimedean condition a necessary structural requirement for the statement that strictly fuzzy $\psi$-contractive implies Cauchy to hold in a GV-fuzzy metric space? That is, does there exist a non-strong GV-fuzzy metric space in which every strictly fuzzy $\psi$-contractive sequence is nevertheless Cauchy?

\section*{Acknowledgements}

The authors would like to thank the reviewers for their valuable comments. This work was supported by the National Natural Science Foundations of China (Grant Nos. 12571489) and Guangxi Natural Science Foundation (No. 2025GXNSFAA069576).

\end{document}